\documentclass[11pt,a4paper]{amsart}
\usepackage{amscd,amssymb,amsopn,amsmath,amsthm,graphics,amsfonts,enumerate,verbatim,calc}
\usepackage[dvips]{graphicx}
\input xy
\xyoption{all}

\usepackage{amssymb,amsmath}

\headsep=1cm \numberwithin{equation}{section}
\newtheorem{theorem}{Theorem}[section]
\newtheorem{lemma}[theorem]{Lemma}
\newtheorem{proposition}[theorem]{Proposition}
\newtheorem{corollary}[theorem]{Corollary}

\theoremstyle{definition}
\newtheorem{definition}[theorem]{Definition}
\theoremstyle{remark}
\newtheorem{remark}[theorem]{Remark}
\newtheorem{example}[theorem]{Example}

\newcommand{\Ass}{\operatorname{Ass}}
\newcommand{\wAss}{\operatorname{wAss}}
\newcommand{\im}{\operatorname{Im}}

\newcommand{\Egrade}{\operatorname{E.grade}}
\newcommand{\Kgrade}{\operatorname{K.grade}}
\newcommand{\Hgrade}{\operatorname{H.grade}}
\newcommand{\cgrade}{\operatorname{c.grade}}
\newcommand{\Cgrade}{\operatorname{\check{C}.grade}}

\newcommand{\Spec}{\operatorname{Spec}}

\newcommand{\rad}{\operatorname{rad}}

\newcommand{\Ext}{\operatorname{Ext}}
\newcommand{\Supp}{\operatorname{Supp}}

\newcommand{\Hom}{\operatorname{Hom}}

\newcommand{\Llm}{\lim\limits}

\newcommand{\coker}{\operatorname{coker}}

\newcommand{\fm}{\frak{m}}
\newcommand{\fp}{\frak{p}}
\newcommand{\fq}{\frak{q}}
\newcommand{\fa}{\frak{a}}
\newcommand{\fb}{\frak{b}}

\begin{document}

\author[Asgharzadeh  and Tousi ]{Mohsen Asgharzadeh and Massoud Tousi}

\title[Grade of ideals with respect to torsion theories ]
{Grade of ideals with respect to torsion theories }

\address{M. Asgharzadeh, School of Mathematics, Institute for
Research in Fundamental Sciences (IPM), P.O. Box 19395-5746, Tehran,
Iran.}
\email{asgharzadeh@ipm.ir}
\address{M. Tousi, Department of Mathematics, Shahid Beheshti University, G. C.,
Tehran, Iran and School of Mathematics, Institute for Research in Fundamental Sciences (IPM),
P.O. Box 19395-5746, Tehran, Iran.}
\email{mtousi@ipm.ir}

\keywords{Grade of an ideal, non-Noetherian rings, torsion theory.\\This research was in part supported by a grant from IPM (No. 89130212)}

\begin{abstract}
In this paper we define and compare different types of the  notion of grade with respect to torsion theories over commutative rings which are not necessarily Noetherian. We do this by using Ext-modules, Koszul cohomology modules, \v{C}ech  and local cohomology modules. An application of these results is
given.
\end{abstract}

\maketitle

\section{Introduction}
There are many definitions of almost zero modules over non-Noetherian rings (see Example \ref{al} below). These classes of almost zero modules are closed under taking submodules, quotients, extensions and closed under taking directed limits. In the literature, such  classes of modules are called \textit{torsion theories}. As references for torsion theory, we refer the reader to \cite {Cal} and  \cite{St}. Our motivation comes from ~\cite[Proposition 1.3]{RSS}, where a connection between almost zero modules and the Monomial Conjecture has been given.

Our aim in this paper is to generalize the theory of grade in two directions. First, classical definitions of grade define it in terms of the vanishing of certain sequences of functors, while the definitions here require that the values of these functors lie in 	torsion theories. The second generalization is from Noetherian to non-Noetherian rings. Grade over not necessarily Noetherian rings was first defined by Barger \cite{B} and Hochster \cite{Ho} (see also \cite{A} and \cite{AT2}). In this paper, we extend the main result of \cite[Section 2]{AT1} by dropping of the Noetherian assumption.

Throughout this paper,  $R$ is a commutative (not necessarily Noetherian) ring, $M$  an $R$-module, $\fa$  a finitely generated ideal of $R$ and $\mathfrak{T}$ a torsion theory of $R$-modules.

We say that a sequence $\underline{x}:=x_{1}, \ldots, x_{r}$ of elements of $R$ is  a weak $M$-regular sequence with respect to $\mathfrak{T}$, if $((x_{1},\ldots, x_{i-1})M:_{M}x_{i})/(x_{1},\ldots,x_{i-1})M\in\mathfrak{T}$  for all $i=1,\ldots,r$. We denote the supremum of the lengths of all weak $M$-regular sequences with respect to $\mathfrak{T}$ which are contained in $\fa$ by $\mathfrak{T}-\cgrade_{R}(\fa,M)$. By using Ext-modules, Koszul cohomology modules, \v{C}ech  and local cohomology modules, we define four types of grade  and we denote these by $\mathfrak{T}-\Egrade_R(\fa,M)$, $\mathfrak{T}-\Kgrade_{R}(\fa,M)$, $\mathfrak{T}-\Cgrade_{R}(\fa,M)$ and $\mathfrak{T}-\Hgrade_{R} (\fa,M)$, respectively. These terms are explained in Definition~\ref{def} below.   It is worth pointing out that different types of the usual notion of grade (over general commutative rings) correspond  to the case $\mathfrak{T}=0$. The following is our first main result (see Theorem~\ref{first} and Proposition~\ref{reg} below):

\begin{theorem}\label{main1}
Let $\mathfrak{T}$ be a  torsion theory of $R$-modules, $\fa$ a finitely generated ideal of $R$ and $M$ an $R$-module. The following holds:
\begin{enumerate}
\item[$\mathrm{(i)}$] $\mathfrak{T}-\Cgrade_{R}(\fa,M)=\mathfrak{T}-\Kgrade_{R}(\fa,M)$ and if $R$ is coherent, then $$\mathfrak{T}-\Kgrade_{R}(\fa,M)=\mathfrak{T}-\Egrade_R(\fa,M)=\mathfrak{T}-\Hgrade_{R}(\fa,M).$$
\item[$\mathrm{(ii)}$]  $\mathfrak{T}-\cgrade_{R}(\fa,M)\leq\mathfrak{T}-\Kgrade_{R}(\fa,M)$.
\end{enumerate}
\end{theorem}

Our methods to prove Theorem \ref{main1} are inspired by \cite[Proposition 1.1.1]{Str} and ideas from \cite{Str},
Chapters 5 and 6. The first interest in Theorem \ref{main1} (at least to the authors) comes from its application in \cite[Proposition 5.4]{AS}. Also, as an immediate application of Theorem \ref{main1}, we give an affirmative answer to \cite[Question 4.8]{AS}. More precisely, let $T$ be an algebra equipped with a value map $v$ over a local ring $(R,\fm)$ and let $M$ be an almost Cohen-Macaulay $T$-module (see Definition \ref{almost}). Suppose $\underline{x}:=x_{1},\ldots,x_{r}$ is a generating set for $\fm$ and denote the $i$-th  \v{C}ech  cohomology module of $M$ with respect to  $\underline{x}$ by $H^{i}_{\underline{x}}(M)$.
Then, in Corollary \ref{ap}, we show that $H^{\dim R}_{\underline{x}}(M)$ is not almost zero. This result extends one  part of \cite[Proposition 4.7]{AS} when it concerned with algebras rather than modules.

In Theorem \ref{main1} $\mathrm{(ii)}$, the equality $\mathfrak{T}-\cgrade_{R}(\fa,M)=\mathfrak{T}-\Kgrade_{R}(\fa,M)$ does not true for general modules, even if $\mathfrak{T}=\{0\}$ (see Example~\ref{fin} (ii), below).
There is a natural question. Under what conditions does the equality hold? Our second aim is to find such conditions both over modules and torsion theories. We consider the \textit{half centered} torsion theories.
Half centered torsion theories are introduced in \cite{Cah1}. Recall that a torsion theory $\mathfrak{T}$ is  called half centered if $M\in \mathfrak{T}$ whenever 
$R/ \fp\in \mathfrak{T}$ for all $\fp\in\wAss_R(M)$, where $\wAss_R(M)$ is the set of all weakly associated prime ideals of $M$. In view of this, half centered torsion theories behave well with respect to  certain cycle submodules of a given module. This allow us to use the induction. In order to perform the inductive step, one assumes the induction hypothesis. So, we need to use the notion of \textit{weakly Laskerian} modules.
 Recall from \cite{DM} that an $R$-module $M$ is called weakly Laskerian, if each quotient of $M$ has finitely many weakly associated prime ideals. Concerning the grade of ideals with respect to half centered torsion theories,
the following is our second main result:

\begin{theorem}\label{main2} (see Theorem~\ref{second})
Let $\mathfrak{T}$ be a  half centered torsion theory, $\fa$ a finitely generated ideal of $R$ and $M$ a weakly  Laskerian $R$-module.
Then $\mathfrak{T}-\cgrade_{R}(\fa,M)=\mathfrak{T}-\Kgrade_{R}(\fa,M)$.
\end{theorem}

\section{Different definitions of grade with respect to torsion theories}

 We begin our work in this paper  by setting notation and recalling some notions.

\begin{definition}\label{tor}
A family $\mathfrak{T}$ of $R$-modules is called a torsion theory, if  $\mathfrak{T}$ is closed under taking submodules,
quotients, extensions and closed under taking directed limits.
\end{definition}

One can easily find that the following classes of almost zero modules are torsion theories.

\begin{example}\label{al}
Let  $R^{+}$ be the integral closure of a Noetherian local domain $(R,\fm)$ in the algebraic closure of its fraction field. It is worth to note that $R^{+}$ is not a Noetherian ring, if $R$ is not a field. Let $M$ be an $R^+$-module.

$(i)$: This is well-known that there exists  a valuation map $v : R^+\longrightarrow\mathbb{Q}\bigcup \{\infty\}$.  In view of \cite[Definition 1.1]{RSS}, an $R^+$-module $M$ is called almost zero with respect to $v$, if for all $m\in M$ and  all $\epsilon>0$ there exists an element $a\in R^{+}$ with $v(a)<\epsilon$ such that $am=0$. We denote for the class of almost zero $R^{+}$-modules with respect to $v$ by $\mathfrak{T}_v$. One can easily find that $\mathfrak{T}_v$ is a torsion theory.

$(ii)$: In view of  \cite[Definition 2.2]{GR}, $M$ is called almost zero with respect to  a maximal ideal $\fm_{R^{+}}$ of $R^{+}$, if $\fm_{R^{+}}M=0$.

$(iii)$: Let $x$ be an  element in the Jacobson radical of $R^{+}$. In view of \cite[Section 2]{F}, $M$ is called almost zero with respect to $x$, if $x^{1/n}$ kills $M$ for arbitrarily large $n$.
\end{example}

\begin{definition}\label{seq} (see \cite[Definition 1.2]{RSS}) Let $\mathfrak{T}$ be a torsion theory and $M$ an $R$-module.
A sequence $\underline{x}:=x_{1}, \ldots, x_{r}$ of elements of $R$ is called a weak $M$-regular sequence with respect to
$\mathfrak{T}$, if $$((x_{1},\ldots, x_{i-1})M:_{M} x_{i})/(x_{1},\ldots,x_{i-1})M\in\mathfrak{T}$$for all $i=1,\ldots,r$.
\end{definition}

Let $\underline{x}:=x_{1},\ldots,x_{r}$ be a finite sequence  of elements of $R$ and let $\check{C}(\underline{x};M)$ be the \v{C}ech complex
of $M$ with respect to $\underline{x}$, i.e., $\check{C}(\underline{x};M)$ is  as follows:
$$
\begin{CD}
 0 @>>> M @>>>\bigoplus_{1\leq i\leq r} M_{x_{i}} @>>>\cdots
@>>>M_{x_1\ldots x_r}@>>> 0.
\end{CD}
$$
We denote the $i$-th cohomology module of $\check{C}(\underline{x};M)$, by $H^{i}_{\underline{x}}(M)$. For an ideal $\fa$ of $R$, by $H^i_{\fa}(M)$, we mean $\underset{n\in\mathbb{N}}{\varinjlim}\Ext^{i}_{R}(R/\fa^n,M)$.

\begin{remark}
Let $\fa$ be a finitely generated ideal of  $R$ with  a generating set $\underline{x}:=x_{1},\ldots,x_{r}$. It is worth to recall from \cite[Theorem 1.1]{Sch} that $H^i_{\fa}(-)$ and $H^{i}_{\underline{x}}(-)$ are not necessarily the same.
\end{remark}

Let $\underline{x}=x_{1},\ldots,x_{r}$ be a finite sequence  of elements of $R$. For an $R$-module $M$, $\mathbb{K}_{\bullet}(\underline{x})$ stands for the Koszul complex of $R$ with respect to $\underline{x}$. By $\mathbb{K}^{\bullet}(\underline{x}; M)$, we mean $\Hom_R(\mathbb{K}_{\bullet}(\underline{x}), M)$. We denote  the $i$-th cohomology   module of  $\mathbb{K}^{\bullet}(\underline{x}; M)$,  by $H^{i}(\underline{x}; M)$. The symbol $\mathbb{N}_{0}$ will denote the set of nonnegative integers.

\begin{lemma}
Let $\fa$ be a finitely generated ideal of  $R$ and $M$ an $R$-module. Suppose
that $\underline{x}:=x_{1},\ldots,x_{r}$ is a generating set for
$\fa$. Then $\inf\{i\in\mathbb{N}_{0}|H^{i}(\mathbb{K}^{\bullet}(\underline{x},M))\notin
\mathfrak{T}\}$  and $\inf\{i\in\mathbb{N}_{0}|H^{i}_{\underline{x}}(M)\notin \mathfrak{T}\}$
do not depend on the choice of the generating sets for $\fa$.
\end{lemma}

\begin{proof}  The former case is in \cite[Lemma 2.5]{AT1}. Note that in that
argument $R$ does not need to be Noetherian. To prove the
later,  note that $H^{i}_{\underline{x}}(M)\cong
H^{i}_{\underline{y}}(M)$, where $\underline{y}$ is a finite
sequence with $\rad(\underline{x}R)=\rad(\underline{y}R)$ (see \cite[Proposition 2.1(e)]{HM}). \end{proof}

\begin{definition}\label{def}
Let $M$ be an $R$-module and $\fa$ a finitely generated ideal of $R$ with  a generating set $\underline{x}:=x_{1},\ldots,x_{r}$. We define \v{C}ech grade, Ext grade, Koszul grade, local cohomology grade and classical grade of $\fa$ on $M$ with respect to $\mathfrak{T}$, respectively, by the following forms:
\[\begin{array}{ll}\mathrm{(i)} &\mathfrak{T}-\Cgrade_{R}(\fa,M):=\inf\{i\in\mathbb{N}_{0}|H_{\underline{x}}^{i}(M)\notin \mathfrak{T}\};\\
\mathrm{(ii)}&\mathfrak{T}-\Egrade_{R}(\fa,M):=\inf\{i\in\mathbb{N}_{0}|\Ext^{i}_R(R/\fa,N)\notin \mathfrak{T}\};\\
\mathrm{(iii)} &\mathfrak{T}-\Kgrade_{R}(\fa,M):=\inf\{i\in\mathbb{N}_{0}|H^{i}(\mathbb{K}^{\bullet}(\underline{x},M))\notin\mathfrak{T}\};\\
\mathrm{(iv)}&\mathfrak{T}-\Hgrade_{R}(\fa,M):=\inf\{i\in\mathbb{N}_{0}|H_{\fa}^{i}(M)\notin \mathfrak{T}\};\\
\mathrm{(v)}&\mathfrak{T}-\cgrade_{R}(\fa,M):=\sup \{\ell\in \mathbb{N}_{0}|\textit{there exists a weak M-sequence
 in } \fa \textit{ with respect } \\& \textit{ to }\mathfrak{T}\textit{ of length } \ell\}.\end{array}\]
Here $\inf$ and $\sup$ are formed in $\mathbb{Z} \cup \{\pm\infty\}$
with the convention that $\inf \emptyset=+ \infty$ and $\sup
\emptyset=-\infty$.
\end{definition}

\section{Proof of Theorem~\ref{main1}}

Theorem~\ref{first} is  our first main  result. To prove it, we need a couple of lemmas.

\begin{lemma}\label{zero}
Let $\mathfrak{T}$ be a  torsion theory,
$\fa$  a finitely generated ideal of $R$ and $M$ an $R$-module. Then
the following assertions are equivalent:
\begin{enumerate}
\item[$\mathrm{(i)}$]$\mathfrak{T}-\Egrade_{R}(\fa,M)=0$;
\item[$\mathrm{(ii)}$]$\mathfrak{T}-\Hgrade_{R}(\fa,M)=0$;
\item[$\mathrm{(iii)}$]$\mathfrak{T}-\Kgrade_{R}(\fa,M)=0$;
\item[$\mathrm{(iv)}$]$\mathfrak{T}-\Cgrade_{R}(\fa,M)=0$.
\end{enumerate}
\end{lemma}
\begin{proof} $(i)\Rightarrow  (ii)$: Clearly, $\Hom_{R}(R/\fa,M)\notin
\mathfrak{T}$. Also, note that $(0:_{M}\fa)\subseteq H^0_{\fa}(M)$. Then in view of $\Hom_{R}(R/\fa,M)\cong (0:_{M}\fa)$, we see that
$H^0_{\fa}(M)\notin \mathfrak{T}$. So
$\mathfrak{T}-\Hgrade_{R}(\fa,M)=0$.

$(ii)\Rightarrow (i)$: Suppose on the contrary that
$\mathfrak{T}-\Egrade_{R}(\fa,M)>0$ and look for a contradiction. By induction on $n$, we show that
$(0:_{M}\fa^{n})\in \mathfrak{T}$. The case $n=1$ follows from
$\mathfrak{T}-\Egrade_{R}(\fa,M)>0$. Suppose, inductively,  we have
established the result for $n$. Note that $\fa^{n}$ is finitely
generated, let $\fa^{n}=(x_{1},\ldots,x_{\ell})R$. The assignment $m\mapsto f(m):=(x_{1}m,\ldots,x_{\ell}m)$ defines
the following $R$-homomorphism:
$$f:(0:_{M}\fa^{n+1})\longrightarrow\bigoplus_{i=1}^{\ell}(0:_{M}\fa^{n}).$$ Consider the following exact
sequence$$
\begin{CD}
0 @>>> (0:_{M}\fa^{n}) @>>>(0:_{M}\fa^{n+1}) @>>> \im f@>>>0.
\end{CD}
$$
It yields that $(0:_{M}\fa^{n+1})\in \mathfrak{T}$. Therefore, $H^0_{\fa}(M)=\bigcup_{n\in \mathbb{N}}
(0:_{M}\fa^{n}) \in \mathfrak{T}$.

$(iii) \Leftrightarrow (i)$ and $(iv) \Leftrightarrow (ii)$ are
trivial.   Note that if $\underline{x}=x_1,\ldots,x_s$  is a generating  set for $\fa$,
then $H^{0}(\underline{x},M)=(0:_M\fa)$ and  $H^0_{\underline{x}}(M)=H^{0}_{\fa}(M)$.  \end{proof}

\begin{definition}
Let $\{F^{i}\}_{i\geq0}$ be a negative strongly  connected sequence of covariant
functors (see \cite[Page 212]{Rot} and \cite[Definition 1.3.1]{BS}).
\begin{enumerate}
\item[$\mathrm{(i)}$]We say that $\{F^{i}\}_{i\geq0}$ has \textit{$\mathfrak{T}$-restriction} property,
if $F^{i}(M)\in \mathfrak{T}$ for all $M\in \mathfrak{T}$ and $i\geq0$.
\item[$\mathrm{(ii)}$] Grade of an $R$-module $M$  with respect to $\{F^{i}\}_{i\geq0}$
is defined by $$F^{-}(M):=\inf\{i\in\mathbb{N}\cup\{0\}|F^{i}(M)\notin\mathfrak{T}\}.$$
\end{enumerate}
\end{definition}

In the absence of the Noetherian assumption on the base ring we  use the following acyclicity Lemma which is motivated by
a famous Lemma of Peskine and Szpiro \cite{PS}.
\begin{lemma}\label{PS}
Let $\{F^{i}\}_{i\geq0}$ be a negative strongly  connected sequence of covariant
functors with $\mathfrak{T}$-restriction property. Let$$
\begin{CD}
\mathrm{C}_{\bullet}:0 @>{d_{s+1}}>> C_{s}@>{d_{s}}>>\cdots @>>>
C_{1} @>{d_{1}}>> C_{0}  @>{d_{0}}>>0
\end{CD}
$$
be a complex of
$R$-modules. If for all $1\leq i\leq s$

\begin{enumerate}
\item[$\mathrm{(i)}$] $F^{-}(C_{i})\geq i$
\item[$\mathrm{(ii)}$]either $H_{i}(\mathrm{C}_{\bullet})\in \mathfrak{T}$
or $F^{-}(H_{i}(\mathrm{C}_{\bullet}))=0$,
\end{enumerate}
then $H_{i}(\mathrm{C}_{\bullet})\in \mathfrak{T}$ for all $1\leq i\leq s$.
\end{lemma}

\begin{proof}
For each $1\leq i\leq s$, set $T_{i}:=\coker d_{i+1}$ and $K_{i}:=\ker d_{i}$.
By decreasing induction on $s$, we show that for each  $0<r\leq s$, we have
$H_{i}(\mathrm{C}_{\bullet})\in
\mathfrak{T}$ and $F^{-}(T_{i})\geq i$ for $0<r\leq i\leq s$.

Let $r=s$. We have $T_s=C_s$. So by our assumption, it is enough to show that
$H_{s}(\mathrm{C}_{\bullet})\in \mathfrak{T}$. Note that $F^0$ is left exact,
 because  $\{F^{i}\}_{i\geq0}$ is a negative strongly  connected sequence of
  covariant functors. This combining with the monomorphism
$H_{s}(\mathrm{C}_{\bullet})\hookrightarrow C_{s}$ implies the following monomorphism:
$$F^{0}(H_{s}(\mathrm{C}_{\bullet}))\hookrightarrow F^{0}(C_{s}).$$In view of  $F^{-}C_{s}\geq s>0$, we have
$F^{0}C_{s}\in \mathfrak{T}$. Hence $F^{0}H_{s}(\mathrm{C}_{\bullet})\in \mathfrak{T}$, and consequently
$F^{-}(H_{s}(\mathrm{C}_{\bullet}))>0$. By using the assumption of  Lemma, we have $H_{s}(\mathrm{C}_{\bullet})\in \mathfrak{T}$.

Now, suppose inductively, that $0<r< s$ and $H_{i}(\mathrm{C}_{\bullet})\in
\mathfrak{T}$ and $F^{-}(T_{i})\geq i$ for $0<r<i\leq s$. In order to use the
induction hypothesis, consider the following two exact sequences:
$$
\begin{CD}
0 @>>> H_{r+1}(\mathrm{C}_{\bullet}) @>>>T_
{r+1} @>>>C_{r}@>>>T_ {r}@>>>0
\\
0 @>>> H_{r+1}(\mathrm{C}_{\bullet}) @>>>T_
{r+1} @>>> K_{r}@>>>H_ {r}(\mathrm{C}_{\bullet})@>>>0.
\end{CD}
$$
So, we have two exact sequences:
$$
\begin{CD}
0 @>>> \overline{T}_{r+1} @>>> C_{r}@>>> T_ {r}@>>>0
\\
0 @>>> \overline{T}_
{r+1} @>>>K_{r} @>>> H_ {r}(\mathrm{C}_{\bullet})  @>>>0,
\end{CD}
$$
where $\overline{T}_ {r+1}$ is the quotient of $T_ {r+1}$ by
$H_{r+1}(\mathrm{C}_{\bullet})$. Also,
$$F^{i}(\overline{T}_ {r+1})\in \mathfrak{T}\Leftrightarrow F^{i}(T_{r+1})\in \mathfrak{T},$$ because $H_{r+1}(\mathrm{C}_{\bullet})\in
\mathfrak{T}$.

From the first short exact sequence we have the following exact sequence: $$ F^{j}(C_{r})\longrightarrow F^{j}(T_{r})
\longrightarrow F^{j+1}(\overline{T}_ {r+1}).$$ Since
$F^{-}(T_{r+1})\geq r+1$ and $F^{-}(C_{r})\geq r$ we find that
$F^{-}(T_{r})\geq r$.

The second short exact sequence induces the following exact sequence:
$$
\begin{CD}
(\ast) \ \:0 @>>> F^{0}(\overline{T}_{r+1}) @>>> F^{0}(K_{r})@>>> F^{0}(H_{r}(\mathrm{C}_{\bullet}))@>>>F^{1}(\overline{T}_ {r+1}).
\end{CD}
$$
We know that $F^{-}(T_{r+1})\geq r+1$ and $F^{-}(C_{r})\geq r$. So, we have
$$F^{1}(\overline{T}_ {r+1})  \textrm{ and }
F^{0}(C_r)\in\mathfrak{T},$$because $r\geq1$.
In view of  the monomorphism $F^{0}(K_{r})\hookrightarrow F^{0}(C_{r})$, we have $F^{0}(K_{r})\in \mathfrak{T}$. From $(\ast)$ we get that $F^{0}(H_{r}(\mathrm{C}_{\bullet}))\in \mathfrak{T}$. By the assumption of Lemma,  $H_{r}(\mathrm{C}_{\bullet})\in \mathfrak{T}$,
which is precisely what we wish to prove.
\end{proof}

Recall that a ring is coherent if each of its finitely generated ideals are finitely presented.

\begin{lemma}\label{res}
Let $\mathfrak{T}$ be a  torsion theory, $\underline{y}:=y_1,\ldots,y_r$ a finite sequence of elements of $R$ and $M$ an $R$-module. Let $\fa:=\underline{y}R$. If
$M\in \mathfrak{T}$, then the following assertions hold:
\begin{enumerate}
\item[$\mathrm{(i)}$] $H^{i}_{\underline{y}}(M)\in \mathfrak{T}$ for all $i$;
\item[$\mathrm{(ii)}$] $\Ext^{1}_{R}(R/\fa,M)\in \mathfrak{T}$;
\item[$\mathrm{(iii)}$] $H^{i}(\underline{y},M)\in \mathfrak{T}$ for all $i$;
\item[$\mathrm{(iv)}$] If $R$ is coherent, then $\Ext^{i}_{R}(R/\fa,M)\in\mathfrak{T}$ for all $i$;
\item[$\mathrm{(v)}$] If $R$ is coherent, then $H^{i}_{\fa}(M)\in\mathfrak{T}$ for all $i$.
\end{enumerate}
\end{lemma}

\begin{proof}
$\mathrm{(i)}$: Let $S$ be a multiplicative closed subset of $R$. One can find from \cite[Section 1.10]{Row} that $S^{-1}M=\Llm_{\stackrel{\longrightarrow}{s\in S}}\Hom_{R}(Rs,M)$. Note that $M\in \mathfrak{T}$. Let $s$ be in $S$. Then $$\Hom_{R}(Rs,M)\cong\Hom_{R}(R/(0:_{R}s),M)\cong(0:_{M}(0:_{R}s))\subseteq M.$$ So $\Hom_{R}(Rs,M)\in \mathfrak{T}$ and consequently ${\varinjlim}_{s\in S}\Hom_{R}(Rs,M)\in\mathfrak{T}$. This shows that $H^{i}_{\underline{y}}(M)\in\mathfrak{T}$, as claimed.

$\mathrm{(ii)}$: Consider the exact sequence $F\to R \to R/\fa\to 0$, where $F$ is a free $R$-module of
finite rank. Such a sequence exists, because $\fa$ is finitely generated. One may find $\Ext^{1}_{R}(R/\fa,M)$ as a subquotient of $\Hom_{R}(F,M)$. Clearly,  $\Hom_{R}(F,M)\in\mathfrak{T}$. So, $\Ext^{1}_{R}(R/\fa,M)\in \mathfrak{T}$.

$\mathrm{(iii)}$: This is straightforward from the definition of Koszul complex and we leave it to the reader.

$\mathrm{(iv)}$: Let $\textbf{F}_\bullet: \cdots \to F_{i+1}\to
F_i\to F_{n-1}\to \cdots\to F_0\to 0$ be a deleted free resolution
of $R/\fa$ consisting of finitely generated free modules (see
\cite[Corollary 2.5.2]{G}). Clearly,
$\Hom_{R}(F_i,M)\in\mathfrak{T}$. Therefore,
$\Ext^{i}_{R}(R/\fa,M)\in \mathfrak{T}$.

$\mathrm{(v)}$: Let $i$ be an integer. Note that $\fa^n$ is finitely generated for all $n$. Then, in view of part $(iv)$, we have $\Ext^{i}_{R}(R/\fa^n,M)\in \mathfrak{T}$ for all $n$. Therefore,$$H^{i}_{\fa}(M):=\underset{n}{\varinjlim}\Ext^{i}_{R}(R/\fa^n, M)\in\mathfrak{T}.$$Note that $\mathfrak{T}$ is closed under taking direct limit.
\end{proof}

\begin{theorem}\label{first}
Let $\mathfrak{T}$ be a  torsion theory of $R$-modules, $\fa$ a finitely generated ideal of $R$ and $M$ an $R$-module. The following assertions hold:
\begin{enumerate}
\item[$\mathrm{(i)}$] $\mathfrak{T}-\Cgrade_{R}(\fa,M)=\mathfrak{T}-\Kgrade_{R}(\fa,M)$;
\item[$\mathrm{(ii)}$] If $R$ is
coherent, then
$\mathfrak{T}-\Kgrade_{R}(\fa,M)=\mathfrak{T}-\Egrade_R(\fa,M)$;
\item[$\mathrm{(iii)}$] If $R$ is
coherent, then
$\mathfrak{T}-\Egrade_R(\fa,M)=\mathfrak{T}-\Hgrade_{R}(\fa,M)$.
\end{enumerate}
\end{theorem}

\begin{proof} $\mathrm{(i)}$: Let $\underline{x}$ be a generating set of $\fa$. First
we show that $$\mathfrak{T}-\Kgrade_R(\fa,M)\leq
\mathfrak{T}-\Cgrade_{R}(\fa,M)=:s \  \ (\ast).$$ Without loss of
generality we can assume that $s<\infty$. Suppose that
$$\mathfrak{T}-\Kgrade_R(\fa,M)\geq s+1.$$ Let
$\mathrm{C}_{t}:0\to C_{{s+1}}\to C_{s}\to \cdots\to C_{0}$ be a
deleted \v{C}ech complex of relative to $M$ where
$C_{{s+1}}:=\bigoplus M_{x_i}$ and set
$F^i(-):=H^{i}(\underline{x},-)$. This is clear that
$\{F^i(-)\}_{i\geq0}$ is a negative strongly  connected sequence of
covariant functors.  Due to Lemma~\ref{res} we know
that$\{F^i(-)\}_{i\geq0}$ has the $\mathfrak{T}$-restriction property.
Now we check the reminder assumptions of Lemma~\ref{PS}. Any torsion
theory is closed under taking localization (see the proof of
Lemma~\ref{res}). Then,  by the commutativity of flat extensions  with
cohomology functors one may easily find that
$$F^-(C_{i})=\mathfrak{T}-\Kgrade_R(\fa,C_{i})\geq
\mathfrak{T}-\Kgrade_R(\fa,M)\geq s+1\geq i$$ for all $1\leq i\leq s+1$. In view of \cite[Proposition 2.1(d)]{HM}, we have  $H^0_{\fa}(H_i(\mathrm{C}_{t}))=H_i(\mathrm{C}_{t})$ for all $i$. Due to Lemma~\ref{zero} we
have either $H_i(\mathrm{C}_{t})\in\mathfrak{T}$ or $F^{-}(H_i(\mathrm{C}_{t}))=0$  for all $i$. Therefore, by Lemma~\ref{PS},
$H_i(\mathrm{C}_{t})\in\mathfrak{T}$ for all $1\leq i\leq s+1$.
On the other hand,
$\mathfrak{T}-\Cgrade_{R}(\fa,M)=s$. So, there exists $1\leq j\leq s+1$ such that $H_i(\mathrm{C}_{t})\notin\mathfrak{T}$. This contradiction completes the
proof of $(\ast)$.

Next we show that $\mathfrak{T}-\Cgrade_{R}(\fa,M)\leq
\mathfrak{T}-\Kgrade_{R}(\fa,M)=:s.$ Without loss of generality we
can assume that $s<\infty$. Suppose that
$$\mathfrak{T}-\Cgrade_{R}(\fa,M)\geq s+1.$$ Let
$\mathrm{C}_{t}:0\to C_{{s+1}}\to C_{s}\to \cdots\to C_{0}$ be  the
deleted Koszul complex, where $C_{{s+1}}=K^0(\underline{x},M)$ and
set $F^i(-):=H^{i}_{\underline{x}}(-)$. Clearly,
$\{F^i(-)\}_{i\geq0}$ is a negative strongly  connected sequence of
covariant functors. By Lemma~\ref{res}, it has
$\mathfrak{T}$-restriction property.  One has
$$F^-(\mathrm{C}_{i})=\mathfrak{T}-\Cgrade_R(\fa,\mathrm{C}_{i})
=\mathfrak{T}-\Cgrade_{R}(\fa,M)\geq s+1\geq i.$$ All the
homology modules of $\mathrm{C}_{\bullet}$ are annihilated by $\fa$. So by Lemma~\ref{zero},
either $H_{i}(\mathrm{C}_{\bullet})\in \mathfrak{T}$ or $F^-(H_{i}(\mathrm{C}_{\bullet}))=0$.
Therefore,  Lemma~\ref{PS} yields that
$H_{i}(\mathrm{C}_{\bullet})\in \mathfrak{T}$ for all $1\leq i \leq s$, which is a
contradiction to $\mathfrak{T}-\Kgrade_{R}(\fa,M)= s$.

$\mathrm{(ii)}$:  Clearly, $\{\Ext^i_R(R/\fa,-)\}_{i\geq0}$ is negative strongly  connected sequence of covariant functors.  Due to our assumptions and Lemma~\ref{res}, we know that $\{\Ext^i_R(R/\fa,-)\}_{i\geq0}$ has $\mathfrak{T}$-restriction
property. Note that the claim $\mathfrak{T}-\Egrade_{R}(\fa,M)\leq\mathfrak{T}-\Kgrade_{R}(\fa,M)=:s$ follows by repeating the proof of part $(i)$.

Now we show that $\mathfrak{T}-\Kgrade_{R}(\fa,M)\leq
\mathfrak{T}-\Egrade_{R}(\fa,M):=s$. Without loss of generality we
can assume that $s<\infty$. Suppose that
$$\mathfrak{T}-\Kgrade_{R}(\fa,M)\geq s+1.$$ Let
$\textbf{F}_\bullet: \cdots \to F_{i+1}\to F_i\to F_{n-1}\to
\cdots\to F_0\to 0$ be a free resolution of $R/\fa$ consisting of
finitely generated free modules (see \cite[Corollary 2.5.2]{G}) and
set $F^i(-):=H^{i}(\underline{x},-)$. Keep in mind that
$\mathfrak{T}$ is closed under finite direct product. For each
$i\geq 0$, set $C^i:=\Hom(F_i,M)$. It consist of  finitely many
direct product of $M$.  One has $F^-(M)\simeq F^-(C^i)$, since
$\mathbb{K}^{\bullet}(\underline{x},\prod M)\simeq \prod
\mathbb{K}^{\bullet}(\underline{x},M)$.  Consider the
$\mathrm{C}_{t}^{\bullet}:0\to C^{0}\to C^{1}\to\cdots\to C^{{s+1}}$
the deleted complex of $\Hom(\textbf{F}_\bullet,M)$. Also, note that
the cohomology modules of $\mathrm{C}_{t}^{\bullet}$ are annihilated by $\fa$.
The reminder of proof is  a repeating the proof of part $(i)$.

$\mathrm{(iii)}$: By making straightforward modification of the proof of $$\mathfrak{T}-\Egrade_R(\fa,M)\leq
\mathfrak{T}-\Kgrade_{R}(\fa,M),$$one can prove that$$\mathfrak{T}-\Hgrade_R(\fa,M)\leq
\mathfrak{T}-\Kgrade_{R}(\fa,M).$$In view of $\mathrm{(ii)}$, we have $$\mathfrak{T}-\Hgrade_R(\fa,M)\leq
\mathfrak{T}-\Egrade_{R}(\fa,M).$$
In order to prove the reverse inequality, let $\underline{x}$ be a generating set for $\fa$ and let $n$ be an integer.
Assume that $\underline{y}$ is a generating set for $\fa^n$. By \cite[Proposition 2.1(e)]{HM}, $H^{i}_{\underline{x}}(M)\cong
H^{i}_{\underline{y}}(M)$, because $\rad(\underline{x}R)=\rad(\underline{y}R)$. Thus $\mathfrak{T}-\Cgrade_{R}(\fa^n,M)=\mathfrak{T}-\Cgrade_{R}(\fa,M)$.
In view of parts $\mathrm{(i)}$ and $\mathrm{(ii)}$, we have
\[\begin{array}{ll}
\mathfrak{T}-\Egrade_R(\fa^n,M)&=
\mathfrak{T}-\Kgrade_{R}(\fa^n,M)\\
&=\mathfrak{T}-\Cgrade_{R}(\fa^n,M)\\
&=\mathfrak{T}-\Cgrade_{R}(\fa,M)\\
&=\mathfrak{T}-\Kgrade_{R}(\fa,M)\\
&=\mathfrak{T}-\Egrade_{R}(\fa,M).
\\
\end{array}\]
So $\Ext^i_R(R/\fa^n,M)\in\mathfrak{T}$ for all $i<\mathfrak{T}-\Egrade_{R}(\fa,M)$.
Therefore, $$\mathfrak{T}-\Egrade_{R}(\fa,M)\leq\mathfrak{T}-\Hgrade_R(\fa,M).$$Note that $\mathfrak{T}$ is closed under taking direct limit.
\end{proof}

Here, we recall some notions from \cite{AS}.

\begin{definition}
\label{almost} Let $T$ be an algebra over a Noetherian local ring $(R,\fm)$.
\begin{enumerate}
\item[$\mathrm{(i)}$]  We say that  $T$ 
equipped with a value map $v$, if there exits a map  $v:T\longrightarrow\mathbb{R} \cup \{\infty\}$
satisfying:\\
$\mathrm{(a)}$
$v(ab)=v(a)+v(b)$ for all $a,b \in T$\\
$\mathrm{(b)}$
$v(a+b) \ge \min\{v(a),v(b)\}$ for all $a,b \in T$\\
$\mathrm{(c)}$
$v(a)=\infty$ if and only if $a=0$.
\item[$\mathrm{(ii)}$]Let $T$  and $v$ be as part $(i)$.
A $T$-module $M$ is called \textit{almost zero} with respect to $v$, if $m \in M$ and $\epsilon>0$ are given, then there exists $b \in T$ such that $b \cdot m=0$ and $v(b)<\epsilon$.
\item[$\mathrm{(iii)}$] Let $T$   and $v$ be as part $(i)$. Let $\underline{x}:=x_{1},\ldots,x_{r}$ be a generating set for $\fm$. A $T$-module $N$ is  called \textit{almost Cohen-Macaulay} over $R$, if $H^{i}_{\underline{x}}(N)$ is almost zero  with respect to $v $ for all $i \ne \dim R$, but $N/\fm N$ is not almost zero  with respect to $v$.\end{enumerate}
\end{definition}

As an immediate application of  Theorem \ref{first}, in the following we present an affirmative answer to \cite[Question 4.8]{AS}:

\begin{corollary}
\label{ap}
Let $T$ be an algebra equipped with a value map $v$ over a Noetherian local ring $(R,\fm)$, $M$  an almost Cohen-Macaulay $T$-module and let $\underline{x}:=x_{1},\ldots,x_{r}$ be a generating set for $\fm$. Then  $H^{\dim R}_{\underline{x}}(M)$ is not almost zero.
\end{corollary}

\begin{proof}
 We denote the class of almost zero modules with respect to $v$
by $\mathfrak{T}_v$. Clearly, $\mathfrak{T}_v$ is a torsion theory.
Also,  $$\mathfrak{T}_v-\Cgrade_{T}(\fm T,M)\geq\dim R,$$ since $H^{i}_{\underline{x}}(M)$ is almost zero for all $i \ne \dim R$.
In the light of Theorem \ref{first}, we see that $$\mathfrak{T}_v-\Kgrade_{T}(\fm T,M)\geq\dim R.$$ Keep in mind that
$M/\fm M$ is not almost zero. So $$\mathfrak{T}_v-\Kgrade_{R}(\fm T,M)=\dim R.$$ Again by applying Theorem \ref{first}, we see that
$$\mathfrak{T}_v-\Cgrade_{T}(\fm T,M)=\dim R.$$ Therefore, $H^{\dim R}_{\underline{x}}(M)\notin\mathfrak{T}_v$, which is the desired claim.
\end{proof}

\begin{proposition}
\label{reg}
Let $\mathfrak{T}$ be a  torsion theory,
$\fa$  a finitely generated ideal of $R$ and $M$ an $R$-module. If $x\in\fa$ is a weak
$M$-regular sequence with respect to $\mathfrak{T}$, then $$\mathfrak{T}-\Kgrade_{R}(\fa,M)=
\mathfrak{T}-\Kgrade_{R}(\fa,M/xM)+1.$$In particular, $\mathfrak{T}-\cgrade_{R}(\fa,M)\leq\mathfrak{T}-\Kgrade_{R}(\fa,M)$.
\end{proposition}

\begin{proof}
Assume that $\underline{x}$ is a generating set for
$\fa$. Consider the following exact sequences
$$
\begin{CD}
0 @>>> (0:_{M} x) @>{\rho}>>  M @>{\pi}>> M/(0:_{M} x)@>>> 0,
\\
0 @>>>M/(0:_{M} x) @>{f}>>  M @>>> M/xM @>>> 0,\\
\end{CD}
$$where $f(m)=xm$ for all $m\in M$.
Then, we have  the following induced exact sequences: $$(\ast)\   \
H^{i}(\underline{x},(0:_{M} x))\longrightarrow
H^{i}(\underline{x},M)\stackrel{H^{i}(\underline{x},\pi)}\longrightarrow
H^{i}(\underline{x},M/(0:_{M}
x))\stackrel{\Delta^{i}}\longrightarrow
H^{i+1}(\underline{x},(0:_{M} x))$$
$$(\ast\ast)\   \  H^{i}(\underline{x},M/(0:_{M} x))
\stackrel{H^{i}(\underline{x},f)}\longrightarrow
H^{i}(\underline{x},M)\longrightarrow
H^{i}(\underline{x},M/xM)\stackrel{\Delta^{i'}}\longrightarrow
H^{i+1}(\underline{x},M/(0:_{M} x)).$$

By our assumption, $(0:_{M} x)\in\mathfrak{T}$. So it follows from
Lemma~\ref{res} that $H^{i}(\underline{x},(0:_{M} x))\in \mathfrak{T}$ for all $i$.
Therefore by $(\ast)$, we have
$$H^{i}(\underline{x},M)\in \mathfrak{T}\Longleftrightarrow H^{i}(\underline{x},M/(0:_{M} x))\in \mathfrak{T}$$
for all $i$.

Set $t:=\mathfrak{T}-\Kgrade_{R}(\fa,M)$ and let $i<t-1$. Due
to $(\ast\ast)$ we have  $H^{i}(\underline{x},M/xM)\in \mathfrak{T}$,
and so$$\mathfrak{T}-\Kgrade_{R}(\fa,M/x  M)\geq t-1.$$
It is enough for us to show that
$H^{t-1}(\underline{x},M/xM)\notin\mathfrak{T}$. In order to prove this, we first
note that:
\[\begin{array}{ll}
H^{t}(\underline{x},f)oH^{t}(\underline{x},\pi)&=H^t(\underline{x},fo\pi)
\\&=x1_{H^t(\underline{x},M)}\\&=0,
\\
\end{array}\]
since $H^{t}(\underline{x},M)$ is  annihilated by $(\underline{x})R$.
Therefore,
$$\im H^{t}(\underline{x},\pi)\subseteq \ker
H^{t}(\underline{x},f)=\im\Delta^{t-1'} \  \
(\ast\ast\ast).$$

Keep in mind that $H^{t}(\underline{x},(0:_{M} x))\in\mathfrak{T}$ and
$H^{t}(\underline{x},M)\notin\mathfrak{T}$. This combining with
$(\ast)$ shows that $\im H^t(\underline{x},\pi)\notin\mathfrak{T}$. Consequently,  $(\ast\ast\ast)$ implies that
$\im\Delta^{t-1'}\notin\mathfrak{T}$. Thus by inspection of  $(\ast\ast)$, we see that $H^{t-1}(\underline{x},M/xM)\notin\mathfrak{T}$.

The second assertion follows from the first.
\end{proof}

\begin{corollary}
\label{add}
Let $\mathfrak{T}$ be a  torsion theory, $\underline{x}:=x_1,\ldots,x_n$ a finite sequence of elements of $R$
and $M$ an $R$-module.  Then the following assertions are equivalent:\begin{enumerate}
\item[$\mathrm{(i)}$] $\underline{x}$ is a weak $M$-regular sequence with respect to
$\mathfrak{T}$;
\item[$\mathrm{(ii)}$] $\mathfrak{T}-\Kgrade_{R}((x_1,\ldots,x_i)R,M)\geq i$ for all $i$.
\end{enumerate}
\end{corollary}

\begin{proof}
In view of Proposition~\ref{reg}, the only non trivial implication is $\mathrm{(ii)}\Rightarrow \mathrm{(i)}$. By induction on $n$,
we show that $\underline{x}$ is a weak $M$-regular sequence with respect to
$\mathfrak{T}$. When $n=1$, we have $H^0(x_1,M)\in \mathfrak{T}$, and so $x_1$ is a weak $M$-regular sequence with respect to
$\mathfrak{T}$. Now suppose that $n>1$ and that result has been proved for all sequences of elements of $R$ with length $n-1$.
Let $\underline{x}:=x_1,\ldots,x_n$ be a sequence of elements of $R$ such that $\mathfrak{T}-\Kgrade_{R}((x_1,\ldots,x_i)R,M)\geq i$
for all $i$ and let $\underline{x}':=x_1,\ldots,x_{n-1}$. Then by the inductive hypothesis, $\underline{x}'$ is  a weak $M$-regular
sequence with respect to $\mathfrak{T}$. In view of Proposition~\ref{reg}, $\mathfrak{T}-\Kgrade_{R}((x_1,\ldots,x_n)R,M/\underline{x}' M)\geq 1$.
It turns out that $x_n$ is a weak $M/\underline{x}' M$-regular sequence with respect to $\mathfrak{T}$. So $\underline{x}$ is a weak
$M$-regular sequence with respect to $\mathfrak{T}$.
\end{proof}

\section{Proof of Theorem~\ref{main2}}

Let $M$ be an $R$-module. Recall that a prime ideal $\fp$ is \textit{weakly associated}
to $M$, if $\fp$ is minimal over $(0 :_R m)$ for some $m\in M $. We
denote the set of all weakly associated prime ideals of $M$ by $\wAss_R(M)$.

\begin{lemma}\label{key}Let $\mathfrak{T}$ be a  torsion theory and $M$ an $R$-module. The following assertions hold:
\begin{enumerate}
\item[$\mathrm{(i)}$] If $M\in \mathfrak{T}$, then $R/ \fp\in \mathfrak{T}$ for all $\fp\in\Supp_R(M)$;
\item[$\mathrm{(ii)}$] If $M\in \mathfrak{T}$, then $R/ \fp\in \mathfrak{T}$ for all $\fp\in\wAss_R(M)$.
\end{enumerate}
\end{lemma}

\begin{proof} This is in \cite{Cah1}. However, we indicate a short proof of it for the convenience of the reader.
\begin{enumerate}
\item[$\mathrm{(i)}$] Let
$\fp\in \Supp_{R}(M)$. There exists  $m\in M$ such that $(0:_{R}m)\subseteq \fp$. Note that $R/(0:_{R}m)\cong Rm \subseteq M$,
it turns out that $R/(0:_{R}m)\in \mathfrak{T}$. The natural
epimorphism $R/(0:_{R}m)\longrightarrow R/\fp$, shows that $R/\fp\in \mathfrak{T}$.
\item[$\mathrm{(ii)}$] This follows by  part $\mathrm{(i)}$, if we apply the fact that $\wAss_R(M)\subseteq\Supp_R(M)$.
\end{enumerate}
\end{proof}

\begin{definition}\label{def1} Denote  the category of $R$-modules and $R$-homomorphism by $R\textbf{-Mod}$. Let $\mathcal{S}\subseteq R\textbf{-Mod}$ be a family  of $R$-modules and $M$ an $R$-module.
\begin{enumerate}
\item[$\mathrm{(i)}$] $\mathcal{S}$ is called a Serre class, if it is closed under taking submodules, quotients and extensions.
\item[$\mathrm{(ii)}$] Recall from \cite{Cah1} that a Serre class is called \textit{half centered}, if it satisfies in the converse of Lemma~\ref{key} (ii).
\end{enumerate}
\end{definition}

It is worth to recall that  torsion theories over Noetherian rings are half centered, see e.g. \cite[Lemma 2.1]{AT2}.
 The following example indicates that the torsion theory of
almost zero modules with respect to a valuation is not half centered.

\begin{example}\label{exam}
Let $(R,\fm)$ be a Noetherian complete local domain  which is not a
field and let $R^{+}$ denote the integral closure of
$R$ in the algebraic closure of its fraction field. It is proved that
$R^{+}$ is quasilocal. We denote its unique maximal ideal by $\fm_{R^{+}}$.
It is well-known that $R^{+}$ equipped with a value map
$v:R^{+}\to\mathbb{Q} \cup \{\infty\}$ satisfying the following conditions:
\begin{enumerate}
\item[$\mathrm{(i)}$]
$v(ab)=v(a)+v(b)$ for all $a,b \in R^{+}$;
\item[$\mathrm{(ii)}$]
$v(a+b) \ge \min\{v(a),v(b)\}$ for all $a,b \in R^{+}$;
\item[$\mathrm{(iii)}$]
$v(a)=\infty$ if and only if $a=0$.
\end{enumerate}
Also, $v$ is nonnegative on $R^+$ and positive on
$\fm_{R^+}$. In view of \cite[Definition 1.1]{RSS}, an $R^+$-module
$M$ is called almost zero with respect to $v$ if for all $m\in M$
and all $\epsilon
>0$ there is an element $a\in R^{+}$ with $v(a)<\epsilon$ such that
$am=0$.  We denote the class of almost zero modules with respect to $v$
by $\mathfrak{T}_v$. Clearly, $\mathfrak{T}_v$ is a torsion theory.
Now consider the following easy facts:

\begin{enumerate}
\item[$\mathrm{(a)}$] Let $0\neq\fp\in\Spec
R^{+}$. We show that $R^{+} / \fp \in \mathfrak{T}_v$.
Indeed, let $x\in \fp$. For any positive
integer $n$ set $f_n(X):=X^n- x\in R^+[X]$. Let $\zeta_n$ be a root
of $f_n$ in $R^+$. It follows that $\zeta_n\in \fp$, since
$(\zeta_n)^n=x\in\fp$. Keep in mind that $v$ is positive on $\fp$.
The equality $v(\zeta_n)=v(x)/n$ indicates that $\fp$ has elements
of small order. Thus, $R^{+} / \fp\in\mathfrak{T}_v$.
\item[$\mathrm{(b)}$] Let $\epsilon \in \mathbb{R}_{>0}$
and set $\fb_{\epsilon}:=\{x\in R^{+}| v(x)\geq \epsilon\}$.
Then  we claim that $R^{+}/\fb_{\epsilon}\notin\mathfrak{T}_v$. Indeed,
it is easy to see that $\fb_{\epsilon}$ is a nonzero ideal of $R^+$.
Clearly, annihilators of  $1+\fb_{\epsilon}$ consist of
elements of order greater or equal than $\epsilon$.
This shows that $R^{+}/\fb_{\epsilon}\notin\mathfrak{T}_v$.
\end{enumerate}
In view of $\mathrm{(a)}$, we have $R^{+} / \fp \in \mathfrak{T}_v$ for all $\fp\in\Supp_{R^{+}}(R^{+}/\fb_{\epsilon})$.
But by $\mathrm{(b)}$, $R^{+}/\fb_{\epsilon}\notin\mathfrak{T}_v$.
Therefore, $\mathfrak{T}_v$ is not half centered.
\end{example}

\begin{definition}\label{1} For an $R$-module $L$ we denote
\begin{enumerate}
\item[$\mathrm{(i)}$] $\{\fp \in \Supp_R (L) |R/\fp \notin \mathfrak{T}\}$ by $\mathfrak{T}-\Supp_{R}(L)$; and
\item[$\mathrm{(ii)}$] $\{\fp\in \wAss_R (L) |R/\fp \notin \mathfrak{T}\}$ by $\mathfrak{T}-\wAss_{R}(L)$.
\end{enumerate}
\end{definition}

\begin{lemma}\label{key1}
Let $\mathfrak{T}$ be a  torsion theory and $M$ an $R$-module. Consider the following conditions:
\begin{enumerate}
\item[$\mathrm{(a)}$] $x_{i}\notin \bigcup_{\fp\in
\mathfrak{T}-\wAss_{R}(M/(x_{1},\ldots,x_{i-1})M)}\fp$ for all $i=1,
\ldots, r$;
\item[$\mathrm{(b)}$] The sequence  $x_{1}, \ldots, x_{r}$ is a weak $M$-regular sequence with
respect to $\mathfrak{T}$;
\item[$\mathrm{(c)}$] For any $\fp \in \mathfrak{T}-\Supp_{R}(M)$, the elements
$x_{1}/1, \ldots, x_{r}/1 $  of the local ring $R_{\fp}$ form a weak
$M_{\fp}$-sequence.
\end{enumerate}
 Then the following assertions hold:
\begin{enumerate}
\item[$\mathrm{(i)}$] The implications
$(c)\Leftrightarrow(a)$ and $(b)\Rightarrow (c)$ are hold.
\item[$\mathrm{(ii)}$]
If $\mathfrak{T}$ is half centered, then $(a)\Rightarrow (b)$.\end{enumerate}
\end{lemma}

\begin{proof}
First recall the following facts of an $R$-module $K$ from \cite[Theorem 3.3.1]{G}:
\begin{enumerate}
\item[$\mathrm{(1)}$]  The set of
zero-divisor of $K$ is $\bigcup_{\fp\in\wAss_R(K)}\fp$ ;
\item[$\mathrm{(2)}$]
$\fp\in \wAss_R (K)$ if and only if $\fp R_{\fp}\in\wAss_{R_{\fp}} (K_{\fp})$.
\end{enumerate}

Now we prove the Lemma.

\begin{enumerate}
\item[] $(b)\Rightarrow (c)$: Let $\fp\in \mathfrak{T}-\Supp_{R}M$. In view of Lemma~\ref{key},
we see
that $$(\frac{((x_{1},\ldots,x_{i-1})M:_{M}x_{i})}{(x_{1},\ldots,x_{i-1})M})_{\fp}=0$$ for all $1\leq i\leq r$. This yields the claim.

\item[] $(c)\Rightarrow (a)$: Let $1\leq i\leq r$ and set $L:=M/(x_{1},\ldots,x_{i-1})M$. Suppose $\fp\in\mathfrak{T}-\wAss_{R}(L)$. Then by
$(2)$, we see that $\fp R_{\fp}\in\wAss_{R_{\fp}}(L_{\fp})$. This along with
$(1)$ and our assumption yield that $x_i/1\notin\fp R_{\fp}$. So $x_i\notin \fp$.
\item[] $(a)\Rightarrow (c)$: Let $\fp\in\mathfrak{T}-\Supp_{R}(M)$. In view of  $(1)$, it is enough to show that $x_i/1\notin\fq R_{\fp}$ for all
$\fq R_{\fp}\in\wAss_{R_{\fp}}(M_{\fp}/(x_{1},\ldots,x_{i-1})M_{\fp})$.

Let $\fq R_{\fp}\in\wAss_{R_{\fp}}(M_{\fp}/(x_{1},\ldots,x_{i-1})M_{\fp})$. Then $\fq\subseteq \fp$ and $(2)$ implies that $\fq \in\wAss_{R}(M/(x_{1},\ldots,x_{i-1})M)$.
We have $R/ \fq\notin\mathfrak{T}$, because $R/ \fp\notin\mathfrak{T}$. By applying our assumption, we get $x_i\notin \fq$. This yields that $x_i/1\notin\fq R_{\fp}$, as claimed.
\item[] $(a)\Rightarrow (b)$: This is in  \cite[Lemma 2.3]{AT1}.
\end{enumerate}
\end{proof}

\begin{lemma}\label{ext}
Let $\mathfrak{T}$ be a  half centered torsion theory, $\fa$
a finitely generated ideal of $R$  with a generating set
$\underline{y}:=y_{1},\ldots,y_{s}$ and $M$ an $R$-module.
  Then the following assertions hold:
\begin{enumerate}
\item[$\mathrm{(i)}$] Let
$\underline{x}:=x_{1}, \ldots, x_{r}$ be a weak $M$-regular sequence with
respect  to $\mathfrak{T}$ in $\fa$. Then $H^{i}(\underline{y}, M)\in \mathfrak{T}$ for all $0\leq i\leq
r-1$. Also, $H^{r}(\underline{y}, M)\notin \mathfrak{T}$ if and
only if $H^{0}(\underline{y}, M/\underline{x}M)\notin \mathfrak{T}$.
\item[$\mathrm{(ii)}$] $\mathfrak{T}-\Kgrade_{R}(\fa,M)=\inf\{\Kgrade_{R_{\fp}}(\fa
R_{\fp},M_{\fp})|\fp\in\mathfrak{T}-\Supp_{R}(M)\}$.
\end{enumerate}
\end{lemma}

\begin{proof} $\mathrm{(i)}$: This is in the proof of \cite[Lemma 2.4]{AT1}. Note that in that
argument $M$ does not need to be finitely generated. Also, recall from \cite[Theorem 1.6.16]{BH} and \cite[Proposition 1.6.10(d)]{BH} that,
if $\underline{w}:=w_{1}, \ldots, w_{\ell}$
is a sequence in $R$ and $(w_{1}, \ldots, w_{\ell})R$ contains a weak $M$-regular sequence $\underline{z}:=z_{1}, \ldots, z_{t}$, then
$H^i(\underline{w},M)=0$ for $0\leq i<t$ and $H^t(\underline{w},M)\cong H^0(\underline{w},M/\underline{z}M)$.

$\mathrm{(ii)}$: This is in the proof of \cite[Proposition  2.7]{AT1}. Note that  Koszul cohomology modules are commute with the localization \cite[Proposition 1.6.7]{BH}.
\end{proof}

Let $M$ be an $R$-module. Recall from \cite{DM} that $M$ is called \textit{weakly Laskerian}, if each quotient of $M$ has finitely many weakly associated prime ideals.
Now we are ready to prove Theorem~\ref{main2}.

\begin{theorem}\label{second}
Let $\mathfrak{T}$ be a  half centered torsion theory, $\fa$ a finitely generated ideal of $R$ and $M$ a weakly  Laskerian $R$-module.
Then $\mathfrak{T}-\cgrade_{R}(\fa,M)=\mathfrak{T}-\Kgrade_{R}(\fa,M)$.
\end{theorem}

\begin{proof}
First note that  by Proposition~\ref{reg}, $\mathfrak{T}-\cgrade_{R}(\fa,M)\leq\mathfrak{T}-\Kgrade_{R}(\fa,M)$. In order
to prove the reverse inequality and without loss of generality, we may and do assume that $n:=\mathfrak{T}-\cgrade_{R}(\fa,M)<\infty$.
Let $\underline{x}:=x_1,\ldots, x_n$ be the supremum of the lengths of all weak $M$-regular sequences with respect to $\mathfrak{T}$
which are contained in $\fa$. It follows from the maximality of $\underline{x}$ and Lemma~\ref{key1} that $\fa \subseteq\fp$ for some
 $\fp\in \mathfrak{T}-\wAss_{R}(M/\underline{x}M)$. Keep in mind that $\fa$ is finitely generated. In particular, $R/ \fa$ is finitely
presented. Thus in view of \cite[Lemma 7.1.6]{G}, we see that$$\fp\in \wAss_R (M/ \underline{x}M)\cap\Supp_R (R/ \fa )=
\wAss_R(\Hom_{R}(R/ \fa,M/ \underline{x}M)).$$Lemma \ref{key} implies that $\Hom_{R}(R/{\fa},M/\underline{x}M)\notin\mathfrak{T}$,
and so $H^{0}(\underline{y};M/ \underline{x}M)\notin\mathfrak{T}$, where $\underline{y}:=y_1,\ldots,y_r$ is a generating set for $\fa$.
In view of  Lemma~\ref{ext} $\mathrm{(i)}$, $H^{n}(\underline{y}, M)\notin \mathfrak{T}$. Again by applying Lemma~\ref{ext} $\mathrm{(i)}$,
$H^{i}(\underline{y}, M)\in \mathfrak{T}$ for all $0\leq i\leq n-1$. Therefore, $\mathfrak{T}-\cgrade_{R}(\fa,M)=\mathfrak{T}-\Kgrade_{R}(\fa,M)$.
\end{proof}

\begin{remark}
Let $R$ be a  coherent ring, $\mathfrak{T}$ a  half centered torsion theory and $M$ a  finitely presented $R$-module.
Let $\underline{x}:=x_1,\ldots,x_n$ be a sequence of elements of $R$ with the property that $\mathfrak{T}-\Kgrade_{R}((x_1,\ldots,x_n)R,M)\geq n$.
Here we show that $\underline{x}$ is a weak $M$-regular sequence with respect to $\mathfrak{T}$.
Indeed, in view of Corollary ~\ref{add}, it is enough to show that $\mathfrak{T}-\Kgrade_{R}((x_1,\ldots,x_i)R,M)\geq i$ for all $i$.
For each $1 \le i< n$,  set $\underline{x}_{i}:=x_1,\ldots,x_i$. Let $\fp\in\Supp (H^j(\underline{x}_{i},M))$ for $0\leq j< i$. We shall achieve the claim by showing that  $R/\fp\in\mathfrak{T}$. Suppose on the contrary that and $R/\fp\notin\mathfrak{T}$ and look for a contradiction. In view of Lemma~\ref{ext} $(ii)$, we see that
$$\Kgrade_{R_{\fp}}(\underline{x} R_{\fp},M_{\fp}) \geq\mathfrak{T}-\Kgrade_{R}(\underline{x}R,M)\geq n.$$
Since $R_{\fp}$ is coherent and $M_{\fp}$ is finitely presented, by making direct modifications of the proof of \cite[Lemma 3.7]{AT2}, one may find that
the Koszul $($co$)$homology modules of $M_{\fp}$ with respect to $\underline{x}$ are finitely presented. Consider the following long exact sequence:
$$
\begin{CD}
\cdots @>>> H^{k}(\underline{x}_i,M_{\fp}) @>{x_{i+1}}>> H^{k}(\underline{x}_i,M_{\fp})@>>>H^{k+1}(\underline{x}_{i+1},M_{\fp}) @>>> \cdots.\\
\end{CD}
$$
Nakayama's lemma yields that $\Kgrade_{R_{\fp}}(\underline{x}_{i+1}R_{\fp},M_{\fp})\le \Kgrade_{R_{\fp}}(\underline{x}_{i}R_{\fp},M_{\fp})+1$. By an easy induction,
$\mathfrak{T}-\Kgrade_{R_{\fp}}(\underline{x}_{i}R_{\fp},M_{\fp}) \geq i$.
Hence, $H^j(\underline{x}_{i},M_{\fp})=0$.
Therefore, $\fp\notin\Supp (H^j(\underline{x}_{i},M))$, a contradiction.
\end{remark}

It is noteworthy to remark that the assumptions of Theorem~\ref{main1}  and Theorem~\ref{main2} are really needed.

\begin{example}\label{fin}$(i)$: Let $(R,\fm)$ be a $1$-dimensional complete local domain of prime characteristic and let  $R^{+}$ be the integral closure of $(R,\fm)$ in the algebraic closure of its fraction field. Recall from \cite[Lemma 5.3]{As} that $R^{+}$ is a valuation domain. Let $\fb\neq R^{+}$ be a nonzero finitely generated ideal of $R^{+}$. By \cite[Corollary 6.9]{As}, $\Ass_{R^{+}}(R^{+}/\fb)=\emptyset$. Thus $R^{+}/\fm_{R^{+}}$ can not be embedded in $R^{+}/\fb$. It turns out that $\Hom_{R^{+}}(R^{+}/\fm_{R^{+}},R^{+}/\fb)=0$, i.e., $\Egrade_{R^{+}}(\fm_{R^{+}},R^{+}/\fb)>0$. Clearly, $\dim(R^{+}/\fb)=0$.  In view of \cite[Lemma 3.2]{AT2}, we have $\Kgrade_{R^{+}}(\fm_{R^{+}},R^{+}/\fb)\leq \dim(R^{+}/\fb)$. Therefore, the claim $$\Kgrade_A(\fa,A)=\Egrade_A(\fa,A)$$does not true for infinitely generated ideals $\fa$ of non-Noetherian rings, even if the ring is coherent and regular.

$(ii)$:  Let $R:=\mathbb{F}[[X,Y]]$, where $\mathbb{F}$ is a field and
set $M:=\bigoplus_{0\neq r\in(X,Y)} R/rR$. As classified by
\cite[Page 91]{Str},  we know that $\Egrade_R(\fm,M)=1$ and
$\cgrade_R(\fm,M)=0$. Therefore, $\mathfrak{T}-\cgrade_{R}(\fa,M)=\mathfrak{T}-\Kgrade_{R}(\fa,M)$ does not true for general modules, even if $\mathfrak{T}=\{0\}$ and the ring is Noetherian.
\end{example}

%%%%%%%%%%%%%%%%%%%%%%%%%%%%%%%%%%%%%%%%%%%%%%%%%%%

\end{document}